\documentclass[11pt]{amsart}

\usepackage[mathscr]{eucal}
\usepackage{amssymb, amsmath,array, amscd}
\usepackage{enumerate}
\usepackage{graphicx}
\usepackage{url}
\usepackage[colorlinks,plainpages]{hyperref}
\usepackage[all]{xy}

\newtheorem{thm}{Theorem}
\newtheorem{defn}[thm]{Definition}
\newtheorem{lem}[thm]{Lemma}
\newtheorem{cor}[thm]{Corollary}
\newtheorem{proposition}[thm]{Proposition}

\theoremstyle{definition}

\begin{document}

\title{The Only Complex 4-Net Is the Hesse Configuration} 

\author{Alp Bassa}
\address{%
Bo\u gazi\c ci University,
Department of Mathematics,
34342 Istanbul,
Turkey
}
\email{alp.bassa@boun.edu.tr}

\author{Al\.i Ula\c{s} \"Ozg\"ur K\.i\c{s}\.isel}
\address{Middle East Technical University, Department of Mathematics,
06531 Ankara, 
Turkey
}
\email{akisisel@metu.edu.tr}


\begin{abstract} It has been conjectured that the only nets realizable in $\mathbb{CP}^2$ are 3-nets and the Hesse configuration (up to isomorphism). We prove this conjecture.
\end{abstract}

\maketitle

\section{Introduction} 
Nets are certain line arrangements in the projective plane, which naturally occur in the study of resonance varieties, homology of Milnor fibers and fundamental groups of curve complements \cite{LY, FY, DS, ABCAL}. We will work over the field of complex numbers, hence all arrangements in question will be in $\mathbb{CP}^{2}$. In this setting, one of the equivalent definitions of a net is as follows: Let $m\geq 3, d\geq 2$ be integers. An $(m,d)$-net is a pencil of algebraic curves of degree $d$ in $\mathbb{CP}^2$ with a base locus of exactly $d^{2}$ points, which degenerates $m$ times totally into a union of $d$ lines. In this paper we determine all pairs $(m,d)$ for which an $(m,d)$-net exists: The only possible values for the pair $(m,d)$ are $(3,d)$ where $d\geq 2$ can take any value, and $(4,3)$. 

The problem of determination of all possible values of $(m,d)$ was first posed as an open problem by Yuzvinsky in \cite{Yuz1}, where the restriction $m\leq 5$ was proven and examples of $(3,d)$-nets for any $d\geq 2$ were given. In \cite{Sti}, Stipins proved that $m\leq 4$. These two results were strengthened respectively in \cite{PY} and \cite{Yuz2} by taking into consideration the possibility of multiple components in the degenerate fibers. The existence of a $(4,3)$-net, the Hesse configuration, was classically known. Furthermore, it can be proven in several ways that the only $(4,3)$-net up to projective equivalence is the Hesse configuration, for instance see \cite{Sti, Gun}. Partial results for small values of $d$, more specifically the impossibility of a $(4,d)$-net for $4\leq d\leq 6$,  were also obtained \cite{DMWZ}. Morever, results in \cite{KNP} imply that $(4,d)$-nets do not exist for $d\equiv 2\ (\rm{mod}\ 3)$. However, the case $(4,d)$ with $d\geq 7$ and $d\equiv 0,1\ (\rm{mod}\ 3)$ remained open and such nets were conjectured not to exist \cite{Yuz1, Yuz2, Yuz3, Dim, DMWZ, BKN, Urz2, Urz3, DS}. Our results confirm this conjecture. 

The method that we follow is an adaptation of an idea that dates back to  Hirzebruch \cite{hirz} combined with a method of Yuzvinsky \cite{Yuz1} that uses fibrations over $\mathbb {CP}^1$:  Assuming the existence of an arrangement, Hirzebruch constructs an algebraic surface which is a multicyclic cover of $\mathbb{CP}^{2}$ with a certain branching structure along this arrangement. Except certain special configurations that he analyzes separately, his construction produces an algebraic surface of general type and the Bogomolov--Miyaoka--Yau inequality then imposes a restriction on the combinatorial possibilities for the line arrangement. This restriction does not rule out $(4,d)$-nets, therefore cannot directly be used. However in the case of nets, after blowing up $\mathbb{CP}^{2}$ at the points of the net, one obtains a natural fibration over $\mathbb{CP}^{1}$. In \cite{Yuz1}Yuzvinsky uses this fibration to show that $(m,d)$-nets do not exist for  $m\geq 6$. Although not expressed in those terms, his proof can be interpreted as the positivity of the Euler characteristic of a similar multiple cover. Other variants of the multiple cover construction were also successfully used in \cite{Urz1, Urz2, Urz3} and \cite{RU} both for the classification of $3$-nets and for resolving important open problems about the geography of surfaces. Inspired by these results, we investigate the signature of an algebraic surface constructed from a multiple cover of the  projective plane blown-up at the base locus of the net and compute it in two different ways. We show that in the conjecturally impossible cases, the equality of these two computations causes a contradiction.

\section{Nets and pencils of plane curves} 

\begin{defn} 
Suppose $d\geq 3$ and $m\geq 3$. An $(m,d)$-net in $\mathbb{CP}^{2}$ is a collection of $m$ disjoint sets of lines $\mathcal{A}_{1},\ldots,\mathcal{A}_{m}$ and a collection $\mathcal{X}$ of points such that 
\begin{itemize}
\item for every $x\in \mathcal{X}$ and for any $i$ there exists unique line in $\mathcal{A}_{i}$ containing $x$, 
\item for any $i\neq j$ the intersection of any line in $\mathcal{A}_{i}$ and any line in $\mathcal{A}_{j}$ belongs to $\mathcal{X}$ 
\end{itemize}
\end{defn}

It can be easily shown that $|\mathcal{X}|=d^{2}$ and also $|\mathcal{A}_{i}|=d$ for each $i$. 

\begin{proposition} \label{proposition:pencil}
Given an $(m,d)$-net as above, let $f$ and $g$ be degree $d$ forms vanishing precisely on the lines of $\mathcal{A}_{1}$ and on the lines of $\mathcal{A}_{2}$ respectively. Then the pencil $\{\mu f+\lambda g=0\}$, where $[\mu:\lambda]\in \mathbb{P}^{1}$, contains at least $m$ fibers which degenerate totally into lines, and the base locus of this pencil is $\mathcal{X}$. 
\end{proposition}

\begin{proof}
It is clear that $\mathcal{X}$ is contained in the base locus. By the definition of a net, $\{f=0\}$ and $\{g=0\}$ intersect transversally, therefore by Bezout's theorem the intersection contains $d^2$ points, so it must be exactly $\mathcal{X}$. Let $f_{i}$ be the degree $d$ form cutting out the union of lines in $\mathcal{A}_{i}$ (so, $f_{1}=f$ and $f_{2}=g$). The claim that $\{f_{i}=0\}$ is an element of the pencil $\{\mu f+\lambda g=0\}$ follows from Noether's AF+BG theorem. 
\end{proof} 

Let us consider the line arrangement obtained by taking the union of all lines in all $\mathcal{A}_{i}$'s. Then, this is an arrangement of $md$ lines, each of which contains precisely $d$ points from $\mathcal{X}$. 

Let $S$ be the surface obtained by blowing up $\mathbb{CP}^2$ at all points of $\mathcal X$. There is a naturally defined fibration
\begin{eqnarray*}
\varphi:S &\to& \mathbb {CP}^1\\
p\in \{\mu f+\lambda g=0\}&\mapsto &[\mu:\lambda]
\end{eqnarray*}
Proposition \ref{proposition:pencil}  above implies that $\varphi$ has at least $m$ totally degenerate fibers, each corresponding to the union of lines in one of the $\mathcal{A}_{i}$'s. Say that these fibers, each of  which contain $d$ lines, are over $q_1, q_2,\ldots,q_m \in \mathbb{CP}^1$.

Let $E_1, \ldots, E_{d^2}$ be the exceptional fibers of the blow-up. Each $E_i$  intersects precisely one line in $\varphi^{-1}(q_j)$ for each $j$.

\section{An algebraic surface associated to an $(m,d)$-net} \label{sec2}
The fibered surface $\varphi: S\to \mathbb {CP}^1$ has $m$ special fibers $W_j:=\varphi^{-1}(q_j)$, $j=1,\ldots,m$, each fiber $W_j$ consisting of $d$ distinct rational curves given as the strict transforms of the lines in $\mathcal A_j$. We will denote this  arrangement of curves on $S$ by $W$. The arrangement $W$ consists of $md$ rational curves. A point $p\in S$ which lies on $r_p>2$ curves of the arrangement $W$ will be called a multiple intersection point. The set of these points will be denoted by ${\rm mult} W$. The number of points $p\in W$ with $r_p=r$ will be denoted by $t_r$. 
Note that we do this counting on the surface $S$, after we have blown-up $\mathbb {CP}^{2}$ in the $d^2$ points of $\mathcal X$ and hence removed the $m$-fold intersection points among lines from distinct sets $\mathcal A_i$. We denote
\begin{equation}\label{f0f1}
f_0:=\sum_{r\geq 2} t_r, \quad f_1:=\sum_{r\geq 2} r\cdot t_r.
\end{equation}
Hence the number of multiple intersection points is given by
$$\#{\rm mult}W=\sum_{r\geq 3} t_r=f_0-t_2.$$
As we have removed all intersections among lines from distinct sets $\mathcal A_i$, all intersections occur among the $d$ curves within each of the $\mathcal A_i$ in the $m$ corresponding fibers. Counting pairs of curves in two different ways we obtain
$$m\cdot {d \choose 2}=\sum_{r\geq 2}t_r\cdot {r\choose 2}.$$
In particular
\begin{equation}
\label{rsquaretr}
\sum_{r\geq 2}r^2\cdot t_r=md(d-1)+f_1.
\end{equation}

Let $\tau:\hat{S}\to S$ denote the blow-up of the surface $S$ at the multiple intersection points. As we blow up in the multiple points, the reduced total transform of $W$ will only have simple crossings. We will denote the total transform of  $W_j$ by $\hat{W_j}$ and the sum of  $\hat{W_j}$ by $\hat{W}$. Curves intersecting at multiple intersection points will be pulled apart and the intersection point will be replaced by a rational exceptional curve intersecting each of the strict transforms of the intersecting curves in a simple double point. The exceptional divisor introduced at an $r$-fold intersection point will have multiplicity $r$. Hence for each point of multiplicity $r$ we obtain $r$ simple double points formed by intersecting a rational curve of multiplicity $1$ and an exceptional divisor of multiplicity $r$. Denote the exceptional curve introduced by blowing up the multiple point $p$ by  $F_p$. We assume $n$ is a prime, such that $n\equiv 1 \pmod {r_p}$ for all multiplicities $r_p$ with $p\in {\rm mult}W$. Arbitrarily large such $n$ can be found as (by Dirichlet's Theorem) there infinitely many primes of the form $1+a\cdot \prod_{p\in {\rm mult}W}r_p, a  \in \mathbb N$.

We construct a multi-cyclic cover $Y$ of $\hat{S}$ as follows: 

Consider the divisors
$$\hat{D_i}=\hat{W}_{i}+(n-1)\hat{W}_{m},\quad i=1,..,m-1.$$
Note that $\mathcal O(\hat{D_i})\sim n\cdot\tau^*(\varphi^*(\mathcal O(1)))$, therefore each $\hat{D_i}$ is an $n$-divisible effective divisor. This data defines a multi-cyclic Kummer cover $\pi^{\prime}:X\to \hat{S}$ with Galois group $(\mathbb Z/n\mathbb Z)^{m-1}$ of degree $n^{m-1}$ with $n$-fold ramification along each $\hat{W}_{j}$ (see \cite[p.54]{BHH} and \cite{gao}). Let $\overline{X}$ denote its normalization and $Y$ a minimal desingularization of $\overline{X}$. Denote the composed map from $Y$ to $\hat{S}$  by $\pi:Y\rightarrow \hat{S}$. Each point $p\in {\rm mult} W$ of multiplicity $r$ will give rise to $n^{m-2}\cdot r$ isolated singularities of $\overline{X}$, each of which will be isomorphic to a singularity of the form $u^n=x\cdot y^r$ in local coordinates. The $t_2$ double points will give rise to singularities on $\overline{X}$, locally given by $u^n=x\cdot y$. All singularities are of Hirzebruch--Jung type and their total number is equal to $n^{m-2}(f_{1}-t_{2})$. 

A singularity of $\overline{X}$ given in local coordinates by $u^n=x\cdot y^r$ will be resolved through a chain of smooth rational exceptional curves $G_1, G_2,\ldots, G_t$. We have $G_i\cdot G_{i+1}=1$ for $i=1,\ldots, t-1$ and no further intersections. The number of exceptional divisors $t$ and the self-intersections can be computed using the negative continued fraction expansion of $n/(n-r)$: we have
\begin{align*}\frac{n}{n-r}=2-\frac{n-2r}{n-r}=2-\frac{1}{\frac{n-r}{n-2r}}=2-\frac{1}{2-\frac{n-3r}{n-2r}}=2-\frac{1}{2-\frac{1}{\frac{n-2r}{n-3r}}}=\cdots \\
\end{align*}
As $n\equiv 1\pmod r$ by assumption, the negative continued fraction expansion has length $t=(n-1)/r$ and is given by $[2,2,\ldots 2, r+1]$. The self-intersection of the $t$ exceptional divisors are $G_i \cdot G_i=-2$ for $1\leq i\leq t-1$ and $G_t\cdot G_t=-(r+1)$.

For $r=1$ (over each of the $t_2$ double points) the resolution introduces a chain of $n-1$ exceptional divisors, each with self-intersection $-2$.

\section{Computation of the signature from the ramification data}
We closely follow Hirzebruch \cite{hirz}. For a more detailed account see \cite{BHH}.
\subsection{Invariants of $\hat{S}$}
Let $l_{ij}$ be one of the lines in $\mathcal{A}_{i}$ and let $\psi: S\to \mathbb {CP}^2$ denote the blow-up map at the $d^2$ points of $\mathcal X$. Denote the strict transform of  $l_{ij}$ with respect to $\psi$ by $L_{ij}$. The total transform of $l_{ij}$ will be of the form
$$\psi^*(l_{ij})=L_{ij}+E_{k_1}+\cdots+E_{k_d}$$
for some distinct $k_1,\ldots,k_d \in \{1,2,\ldots,d^2 \}$. 

\begin{proposition} \label{prop:squares}
(a) For any fiber $F$ of $\varphi$, we have $F^2=0$. In particular,  we have  $\bigl(\varphi^{-1}(q_i)\bigr)^2=0$ for $i=1,2,\ldots,m$.

(b) $L_{ij}^2=1-d$.
\end{proposition}

\begin{proof} 
(a) is clear. For (b), note that $l_{ij}^2=1, E_k^2=-1$. Therefore,
\begin{align*}
1&=l_{ij}^2=\bigl(\psi^*(l_{ij})\bigr)^2=(L_{ij}+E_{k_1}+\cdots+E_{k_d})^2\\
&=L_{ij}^2+E_{k_1}^2+\cdots+E_{k_d}^2+2\underbrace{L_{ij}\cdot E_{k_1}}_{1}+\ldots+2\underbrace{L_{ij}\cdot E_{k_d}}_{1}+\sum 2\underbrace{E_{k_j}\cdot E_{k_l}}_{0}\\
&=L_{ij}^2-d+2d = L_{ij}^2+d,
\end{align*}
from which the claim follows. 
\end{proof} 

It is well known that $K_{\mathbb {CP}^2}=-3[\ell]$ for a line $\ell \subseteq \mathbb {CP}^2.$
The canonical class of the blow up $S$ then equals 
\begin{equation}
\label{eqthree}
K_{S}=\psi^*K_{\mathbb {CP}^2}+\sum_{i=1}^{d^2}E_i= -3\psi^*([\ell])+\sum_{i=1}^{d^2}E_i
\end{equation}

Recall that $c_2(S)=e( S)$. Let $c_2=\langle c_2( S),[{S}]\rangle$ where $[{S}]$ denotes the fundamental class of ${S}$. Therefore $c_2=\chi( S)$. Since each blow-up increases the Euler characteristic by $1$, and $\chi(\mathbb {CP}^2)=3$, we get
\begin{equation}
\label{eqfour}
c_2=3+d^2.
\end{equation}
Note that $c_1({S})=-K_{{S}}$. Let $c_1^2$ denote $\langle c_1^2({S}),[{S}]\rangle$. Then $c_1^2=K_{{S}}^2$. So
$$c_1^2=\bigl(-3\psi^*([\ell])+\sum_{i=1}^{d^2}E_i\bigr)^2$$
Choosing $\ell$ so that it does not contain any points from $\mathcal X$, and letting $L=\psi^{-1}(\ell)$, we obtain
$$c_1^2=\bigl(-3L+\sum_{i=1}^{d^2}E_i\bigr)^2=9L^2+\sum_{i=1}^{d^2}E_i^2=9-d^2$$
(here we used $E_i\cdot L=0, E_i\cdot E_j=0\text{ for } i\neq j, L^2=1$).
So we have
\begin{equation}
\label{eqfive}
c_1^2=9-d^2.
\end{equation}

$\hat{S}$ is obtained from $S$ by a total of $\#{\rm mult}W=f_0-t_2$ blow-ups, each of which replaces a point $p\in S$ by a rational curve $F_p$, hence increases the Euler characteristic by $1$. Hence we have
$$e(\hat{S})=e(S)+f_0-t_2=3+d^2+f_0-t_2.$$
Similarly $e(W)=md\cdot 2-\sum_{r\geq 2}t_r(r-1)=md\cdot 2-f_1+f_0$ and
$$e(\hat{W})=e(W)+f_0-t_2=md\cdot 2-\sum_{r\geq 2}t_r(r-1)+f_0-t_2.$$

\subsection{Invariants of $Y$}
As $\pi:Y\to \hat{S}$ is a degree $n^{m-1}$ map with ramification index $n$ over points of $\hat{W}$, we obtain
\begin{equation*}e(Y)=n^{m-1}\cdot e(\hat{S}-\hat{W})+e(\pi^{*}(\hat{W}))
\end{equation*}

The total transform $\pi^{*}(\hat{W})$ will consist of $mn^{m-2}$ connected components. However, this term will contribute to the order $n^{m-1}$ term of $e(Y)$, since the resolution of each node will produce exceptional curves, whose number will be at the order of $n$. More precisely, while resolving each of the $r t_r n^{m-2}$ singular points above points of multiplicity $r>2$, we will introduce $(n-1)/r$ rational exceptional curves. Moreover the resolution of each of the $t_2 n^{m-2}$ singular points above the double points gives $n-1$ rational exception curves. In total these add $f_0$ to the coefficient of the $n^{m-1}$ term. We get 
\begin{equation*} e(Y)=n^{m-1} \cdot (3+d^2-2md+f_1) +O(n^{m-2}).
\end{equation*}
Next we compute $K_Y$ and the characteristic number $c_1^2(Y)=K_Y^2$. We have
$$K_{S}= -3\psi^*([\ell])+\sum_{i=1}^{d^2}E_i$$
and 
$$K_{\hat{S}}= \tau^* K_{S}+\sum_{p\in {\rm mult} W}F_p$$
Writing $\hat{W}=\sum_{p\in {\rm mult}W}r_p\cdot F_p+R$, with $R$ the strict transform of $W$ and considering the ramification behavior described above we have the $\mathbb Q$-numerical equivalence (see \cite[p. 85]{Urz3})
$$K_Y\equiv \pi^*\bigl(K_{\hat{S}}+\frac{n-1}{n}(\sum_{p\in {\rm mult}W}F_p+R)\bigr)+\Delta,$$
where $\Delta$ is a $\mathbb Q$-divisor supported on the exceptional locus of the desingularization. The exceptional divisors introduced while resolving points over the double points will have no contribution to $\Delta$, as in these cases the intersecting irreducible components will have the same multiplicity. 

$\Delta$ is supported on the exceptional locus and hence its image under $\pi$ is ${\rm mult} W$. The divisor $K_{\hat{S}}+\frac{n-1}{n}(\sum_{p\in {\rm mult}W}F_p+R)$ on the smooth surface $\hat{S}$ can be moved away from ${\rm mult} W$, so that the pull-back of the resulting divisor does not meet $\Delta$. So
$$\pi^*\bigl(K_{\hat{S}}+\frac{n-1}{n}(\sum_{p\in {\rm mult}W}F_p+R)\bigr)\cdot \Delta=0.$$
 Hence
\begin{equation}\label{ky2_1}
K_Y^2=n^{m-1}\cdot\bigl(K_{\hat{S}}+\frac{n-1}{n}(\sum_{p\in {\rm mult}W}F_p+R)\bigr)^2+\Delta^2.
\end{equation}
We have
\begin{multline*} \bigl(K_{\hat{S}}+\frac{n-1}{n}(\sum_{p\in {\rm mult}W}F_p+R)\bigr)^2\\
=\bigl(-3\tau^*\psi^*([\ell])+\sum_{i=1}^{d^2}\tau^*E_i+\sum_{p\in {\rm mult} W}F_p+\frac{n-1}{n}(\sum_{p\in {\rm mult}W}F_p+R)\bigr)^2.
\end{multline*}

Now 
$$R=\tau^*\psi^*M-m\sum_{i=1}^{d^2}\tau^*E_i-\sum_{p\in {\rm mult}W}r_pF_p,$$
where $M$ is the $(m,d)$-net we started with.
Moreover $\tau^*E_i\cdot F_p=0, E_i\cdot \psi^*M=0$.
So we obtain
\begin{align*}K_Y^2=&n^{m-1}\cdot \Bigl(\bigl(\tau^*\psi^*(-3[\ell]+\frac{n-1}{n}M)\bigr)^2+\bigl((1-m\frac{n-1}{n})\sum_{i=1}^{d^2}\tau^*E_i\bigr)^2\\&+\bigl(\sum_{p\in {\rm mult}W}	(1+\frac{n-1}{n}(1-r_p))F_p\bigr)^2\Bigr)+\Delta^2\\
=&n^{m-1}\cdot \Bigl(\bigl(-3+md\frac{n-1}{n}\bigr)^2-\bigl(1-m\frac{n-1}{n}\bigr)^2d^2\\&-\sum_{r\geq 3}t_r\bigl(1+\frac{n-1}{n}(1-r)\bigr)^2\Bigr)+\Delta^2 \\
=&n^{m-1}(md^{2}-5md-d^2+9+3f_{1}-4f_{0})+\Delta^2+O(n^{m-2}). 
\end{align*}
In the last step, \eqref{f0f1} and \eqref{rsquaretr} were used to simplify the expression. 

Next we show $\Delta^2=O(n^{m-2})$ (we follow the notation in \cite[Appendix]{Urz3}): All components of $\Delta$ lie over the points ${\rm mult}W$ and are supported at the exceptional divisors of the resolutions of singular points $q$ of $\overline{X}$ (there is no contribution from the singularities over double points). Letting $\Delta=\sum_{q} \Delta_q$, we get $\Delta^2=\sum_{q} \Delta_q^2$.
We will first estimate the contribution coming from the resolution of one singular point $q$ of $\overline{X}$
over a point of multiplicity $r$. As seen above the resolution will give rise to a chain of $t=(n-1)/r$ rational exceptional curves $G_1,G_2,\ldots, G_t$ and we have $\Delta_q=\sum_{i=1}^{t}\alpha_i G_i$. The $\alpha_i$ can be calculated using formulas in \cite[Appendix]{Urz3}:  
$$ \alpha_{i}=\frac{i(1-r)}{n},\ {\rm for}\ i=1,\ldots,t.$$
As $G_{i}^{2}=-2$ for $i\leq t-1$, $G_{t}^{2}=-(r+1)$ and  $G_{i}\cdot G_{i+1}=1$ and all other pairs give $0$ intersection, we get, 
\begin{align*} \Delta_q^{2}&=  -2\alpha_{1}^2-2\alpha_{2}^{2}-\ldots-2\alpha_{t-1}^{2}-(r+1)\alpha_{t}^{2}+2\alpha_{1}\alpha_{2}+2\alpha_{2}\alpha_{3}+\ldots+2\alpha_{t-1}\alpha_{t}\\
&=-2\alpha_{1}(\alpha_{1}-\alpha_{2})-2\alpha_{2}(\alpha_{2}-\alpha_{3})-\ldots-2\alpha_{t-1}(\alpha_{t-1}-\alpha_{t})-(r+1)\alpha_{t}^{2} 
\end{align*}
But now $\alpha_{i}-\alpha_{i+1}=\frac{r-1}{n}$ and $|\alpha_{i}|\leq 1$ which shows that $\Delta_q^{2}=O(1)$. As there are a total of $n^{m-2}(f_{1}-2t_{2})$ such singular points $q$, we have $\Delta^2=O(n^{m-2})$. Hence we get
$$K_Y^2=n^{m-1}(md^{2}-5md-d^2+9+3f_{1}-4f_{0})+O(n^{m-2}).$$

Using the expressions for $e(Y)$ and $K_Y^2$ we compute the top order term of the signature $\sigma(Y)$. We obtain the following result: 
\begin{thm}
$$\sigma(Y)=\frac{1}{3}n^{m-1}\Bigl((m-3)d^2-md+3+f_1-4f_0\Bigr)+O(n^{m-2}).$$
\label{thm4}
\end{thm}






\section{Estimating the signature using the fibration over $\mathbb{CP}^1$}
In this section we will estimate the signature of $Y$ by using a different method; we want to make use of the fact that the ramification locus of the cover $\pi$ is trapped in the fibers of $\varphi: \hat{S}\rightarrow \mathbb{CP}^{1}$.  

Let $U_i$ be a sufficiently small open disk in $\mathbb {CP}^1$ centered at $q_i$. Let $T=\varphi^{-1}({\mathbb CP}^1\backslash \cup U_i)$ and $\hat{T}=\tau^{-1}(T)$. Notice that $\hat{T}$ and $T$ are homeomorphic since the blow-ups involved in the map $\tau$ only affect the fibers of $\varphi$ over $q_1,\ldots, q_m$. Recall that the signature of a manifold with boundary can be defined in a way similar to the case of a closed manifold, see \cite{GS, Oz}.
\begin{lem} \label{lem6} The signature  $\sigma(\varphi^{-1}(U_i))$ of the manifold $\varphi^{-1}(U_{i})$ with boundary is equal to $1-d$.
\end{lem}
\begin{proof}
By Lefschetz duality, $H^2(\varphi^{-1}(U_i),\partial( \varphi^{-1}(U_i)))$ is isomorphic to $H_{2}(\varphi^{-1}(U_i))$, which in turn is generated by the classes of the $d$ lines in $\varphi^{-1}(q_i)$, since $\varphi^{-1}(U_i)$ can be retracted to $\varphi^{-1}(q_i)$.  The signature then can be calculated directly, noting that the self-intersection number of any of these lines is $1-d$ and the intersection number of any two distinct lines is $1$.
\end{proof}
By Novikov additivity 
\begin{align*}
\sigma(\hat{T})=\sigma(T)&=\sigma(S)-m\cdot(1-d)\\
&=(1-d^2)-m\cdot(1-d)\\
&=(1-d)(1-m+d).
\end{align*}
Then 
\begin{equation} \label{negative}
-a:=\sigma(\hat{T})<0  \text{ for } d>m-1.
\end{equation}

Notice that $\pi:\pi^{-1}( \hat{T}) \to \hat{T}$ is unramified of degree $n^{m-1}$. In general the signature of a manifold with boundary is not  multiplicative under unramified covers. The discrepancy can be calculated using the $\eta$ invariant appearing in the Atiyah--Patodi--Singer Index Theorem \cite{APS, Ati}. Choose a Riemannian metric on $\hat{T}$ cylindrical near the boundary, in other words such that a neighborhood of its boundary $\partial(\hat{T})$ is isometric to $[0,1]\times \partial(\hat{T})$. Let $p_{1}(\hat{T})$ be the first Pontryagin form obtained from the resulting Riemannian connection. By a special case of the APS Theorem (see \cite{APS}, Theorem 2), we have
$$\sigma(\hat{T})-\int_{\hat{T}}p_1(\hat{T})=\eta(\partial \hat{T}).$$
By using the same theorem for $\pi^{-1}(\hat{T})$ equipped with the pull back metric with respect to $\pi$, we obtain
$$\sigma(\pi^{-1}(\hat{T}))-\int_{\pi^{-1}(\hat T)}p_1(\pi^{-1}(\hat{T}))=\eta(\partial (\pi^{-1}(\hat{T}))).$$
Since the integral of the Pontryagin form is obviously multiplicative under the unramified cover, we obtain
\begin{equation}
\label{etaeq}
\sigma(\pi^{-1}(\hat{T}))-n^{m-1}\sigma({\hat{T}})=\eta(\partial (\pi^{-1}(\hat{T})))-n^{m-1}\eta(\partial \hat{T}).
\end{equation}
A nice feature of this formula is that the two sides are obviously independent of the choice of the Riemannian metric, moreover the right hand side only depends only on the boundary 3-manifold. We will now construct a new 4-manifold with boundary $Z$, cobordant to $\hat{T}$. Then, by making a direct computation on $Z$,  we will show that the right hand side of \eqref{etaeq} is zero in this case.

Recall that $U_1,\ldots U_m$ in ${\mathbb CP^1}$ are  small enough disjoint disks around $q_1,\ldots,q_m$, with boundaries $K_1,\ldots,K_m$, each homeomorphic to a circle. Without loss of generality we may assume that $\varphi$ is a fiber bundle over each $K_i$ with fibers homeomorphic to an oriented genus $g$ surface $\Sigma_g$, where $g=(d-1)(d-2)/2$. We next want to prove that this genus $g$ surface fibration over $K_{i}$ is the boundary of a genus $g$ handlebody fibration over $K_{i}$. 

\begin{lem} 
The fibration $\varphi$ described above lifts to a genus $g$ handlebody fibration $\tilde{\varphi}$. More precisely, there exists a smooth $4$-manifold $Z$ with boundary equal to $\partial{\hat{T}}$ and a map $\tilde{\varphi}:Z\rightarrow \bigcup K_{i}$ such that for every $p\in \bigcup K_{i}$, the fiber $\tilde{\varphi}^{-1}(p)$ is diffeomorphic to a genus $g$ handlebody $N_{g}$. Furthermore, $\partial(\tilde{\varphi}^{-1}(p))=\varphi^{-1}(p)$. 
\end{lem} 

\begin{proof}
It is enough to prove the statement for one connected component $K_{i}$ of the base. Suppose that $p\in K_{i}$ and fix a diffeomorphism $\varphi^{-1}(p)\cong \Sigma_{g}$. Then the monodromy of the fibration $\varphi$ is given by a composition of Dehn twists around the vanishing cycles of the Lefschetz fibration $\varphi: \varphi^{-1}(U_{i})\rightarrow U_{i}$. However, above the central fiber $q_{i}$ the singular fiber is a union of $d$ rational curves, hence it is possible to take all of these vanishing to be pairwise disjoint. We claim that an element of the mapping class group of $\Sigma_{g}$ that is a composition of Dehn twists around pairwise disjoint vanishing cycles lifts to an element of the handlebody mapping class group of $N_{g}$, namely it is possible to construct $N_{g}$ such that $\partial(N_{g})=\Sigma_{g}$ and pairwise disjoint disks in $N_{g}$ whose boundaries are these vanishing cycles, so that the Dehn twists will lift to disk twists: Indeed, cut $\Sigma_{g}$ along the $\leq g$ nonseparating ones among these cycles and with additional nonseparating cycles not intersecting these ones if necessary in order to obtain a double cover of the standard $4g$-gon model with cylindrical ends above $2g$  sides and ramified over the others, placed in an alternating way along the boundary of the $4g$-gon. But then the remaining (separating) vanishing cycles will still be disjoint in this model and not intersecting the boundary, and therefore it is clear that we can fill all the vanishing cycles in by pairwise disjoint disks. By using this claim we can now construct $Z$ and the fibration $\tilde{\varphi}$ as the mapping torus of the resulting element of the mapping class group of $N_{g}$. 
\end{proof}


\begin{lem} The intersection form on $H^{2}(Z, \partial Z, \mathbb{Z})$ is trivial. In particular, $\sigma(Z)=0$.  
\label{lemma6}
\end{lem}

\begin{proof} 
By Lefschetz duality, it is enough to show that the intersection product on $H_{2}(Z, \mathbb{Z})$ is trivial. Each component of $Z$ is homotopy equivalent to a fibration over $S^{1}$ where each fiber is a wedge of $g$ circles. This implies that $H_{2}(Z,\mathbb{Z})$ is a free Abelian group generated by $g$ tori lying on $\partial Z$, each of which is a fibration over $S^{1}$. However, pushing the fibers towards the interior of the handlebody $N_{g}$ at each fiber shows that these homology classes intersect trivially. 

\end{proof}


\begin{cor} \label{cor8}
$\sigma(\pi^{-1}(\hat{T}))=n^{m-1}\sigma({\hat{T}})$,
in particular \\ $\sigma(\pi^{-1}(\hat{T}))=-an^{m-1}\leq-n^{m-1}$ for $d>m-1$.
\end{cor}
\begin{proof}Let $\alpha: S^1\to S^1$ be the $n$-fold covering map and let us denote by $\tilde{Z}$ the $n^{m-2}$ disjoint copies of the pullback of $Z$ under $\alpha$. $\tilde{Z}$ is an $n^{m-1}$-fold unramified covering of $Z$. We have
\begin{align*}
&\sigma(\pi^{-1}(\hat{T}))-n^{m-1}\sigma({\hat{T}})=\eta(\partial (\pi^{-1}(\hat{T})))-n^{m-1}\eta(\partial \hat{T})\\
&=\eta(\partial (\tilde{Z}))-n^{m-1}\eta(\partial Z)\\
&=\sigma(\tilde{Z})-n^{m-1}\sigma(Z)=0
\end{align*}
The last equality is a direct consequence of Lemma~\ref{lemma6}. By Inequality~\ref{negative}, and the fact that $\sigma(\hat{T})=-a$ is an integer, we obtain $$\sigma(\pi^{-1}(\hat{T}))=-an^{m-1}\leq-n^{m-1}.$$
\end{proof}
\begin{thm} $\sigma(Y)=n^{m-1}(-a-\frac{2}{3}f_0)+O(n^{m-2})$ for $d>m-1$. \label{thm9}
\end{thm}
\begin{proof}

Let $V_{i}=\tau^{-1}\bigl(\varphi^{-1}(U_i)\bigr)$ and $V=\bigcup V_{i}$. We can obtain $Y$ by glueing $V$ back to $\pi^{-1}(\hat{T})$. Therefore, by Novikov additivity, it will suffice to show that $\sigma(V)=-\frac{2}{3}n^{m-1}f_0+O(n^{m-2})$. First of all, each $V_{i}$ is contractible to the central fiber over $q_{i}$, hence its second homology is generated by the classes of the irreducible components of the total transform $\pi^{*}(\hat{W_{i}})$. It is enough to consider the contribution of the resolution of the $n^{m-2}(f_{1}-t_{2})$ many Hirzebruch--Jung singularities only, since the total number of all other components of $\pi^{*}(\hat{W_{i}})$, hence their contribution to signature, is $O(n^{m-2})$. 

Therefore, the proof will be completed by estimating the contribution of each singularity to the signature. At a singularity above a point of multiplicity $r>2$ (respectively above a double point), the signature of the intersection matrix of the exceptional curves of the resolution is equal to $-(n-1)/r$ (respectively $-(n-1)$), by explicit computation. The contribution of the singularity to the signature of the manifold, however, must be corrected by the signature defect, which can be computed by a cotangent sum (which can be explicitly deduced from the $G$-signature theorem), see \cite[Theorem p. 225]{hirz2}. For a singularity above a point of multiplicity $r$ the signature defect is given in terms of a Dedekind sum by $4n\cdot s(r,n)$. Using the reciprocity of Dedekind sums, the assumption that $n\equiv 1 \pmod r$ and the well-known trigonometric sum $\sum_{j=1}^{r-1}\cot^2(\frac{\pi\cdot j}{r})=\frac{(r-1)(r-2)}{3}$ (see \cite{BY}), we see that
$$\frac{{\rm def}(n;1,n-r)}{n}=\frac{n^2-n\cdot((r-1)(r-2)+3r)+r^2+1}{3rn},$$
so the contribution of such a singularity to the signature is given by $$-(n-1)/r+n/3r+O(1)=-\frac{2n}{3r}+O(1).$$
The contribution for singularities above double points (taking $r=1$ above) is $-(n-1)+n/3+O(1)=-2n/3+O(1)$. As there are $n^{m-2}rt_r$ singularities over multiplicity $r>2$ points and $n^{m-2}t_2$ singularities above double points, the contribution to the signature from singularities is $-2n^{m-1}f_0/3 +O(n^{m-2})$.
\end{proof}




\section{Implications for the Existence of Nets}

\begin{thm} \label{mainthm}
Let $m\geq 4$. If $d\geq m$ an $(m,d)$-net cannot be realized in $\mathbb{CP}^2$. In particular there is no $(4,d)$-net for $d\geq 4$.
\end{thm}
\begin{proof} Suppose an $(m,d)$-net exists with  $m\geq 4$ and $d\geq m$. By the construction in Section~\ref{sec2} there exists a surface $Y$, whose signature can be computed both as in Theorem~\ref{thm4}  and as in Theorem~\ref{thm9}. Equating the coefficients of $n^{m-1}$ in these quantities, we get 
$$   \frac{1}{3}\Bigl((m-3)d^2-md+3+f_1-2f_0-2f_0\Bigr)=-a-\frac{2}{3}f_0. $$
Noting that $a\geq 1$ and $f_{1}-2f_{0}\geq 0$, this implies 
$$ (m-3)d^2-md+3 =((m-3)d-3)(d-1)< 0. $$ 
However, it is clear that this inequality cannot hold for $m=4, d\geq 4$, or for $m\geq 5, d\geq 3$. This contradiction finishes the proof. 
\end{proof}


As a consequence we obtain the following complete solution to the existence problem of $(m,d)$-nets in $\mathbb{CP}^{2}$.
\begin{thm} 
Let $m\geq 3$, $d\geq 3$. 
\begin{enumerate} 
\item[i)] For $m=3$ and any value of $d$ there exists an $(m,d)$-net.
\item[ii)] For $m=4$ an $(m,d)$-net exists if and only if $d=3$. Furthermore, any $(4,3)$-net is isomorphic to the Hesse configuration.
\item[iii)] For $m\geq 5$ there exists no $(m,d)$-net for any value of $d$.
\end{enumerate}
\end{thm}
\begin{proof}
\emph{i)} The Fermat arrangement of degree $d$ provides an example for a $(3,d)$-net. For further details or other possibilities of $(3,d)$-nets see \cite{Sti, Urz2, Urz3}.

\emph{ii)} The nonexistence of $(4,d)$-nets for $d\geq 4$ follows directly from Theorem~\ref{mainthm}. The uniqueness for $d=3$ has been proven in several places, see for instance \cite{Sti}. For a tropical proof of the uniqueness see \cite{Gun}.

\emph{iii)} If an $(m,d)$-net exists, then deleting one family of lines gives rise to an $(m-1,d)$-net. Therefore by \emph{ii)} an $(m,d)$-net does not exist for $m\geq 4$ and $d\geq 4$. To finish the proof it suffices to show that a $(5,3)$-net does not exist. However this follows immediately from the uniqueness of the $(4,3)$-net up to projective equivalence. 
\end{proof}
\section*{Acknowledgements}
We are grateful to Alexander Degtyarev, Sergey Finashin, Mustafa Hakan G\"unt\"urk\"un, Mustafa Korkmaz, Ferihe Atalan Ozan and Giancarlo A. Urz\'ua for useful discussions and pertinent remarks on previous versions of this paper. We thank the \.{I}stanbul Center of Mathematical Sciences (IMBM) and Nesin Mathematics Village for their support, where different stages of the project were completed. 
\bibliographystyle{99}

\begin{thebibliography}{A}


\bibitem{ABCAL} Artal Bartolo, E. , Cogolludo-Agustin, J. I. , Libgober, A. \emph{Depth of cohomology support loci for quasi-projective varieties via orbifold pencils}, Rev. Mat. Iberoam. 30 (2014), no. 2, 373–404.

\bibitem{Ati} Atiyah, M. F. \emph{The logarithm of the Dedekind $\eta$-function}, Math. Ann. 278 (1987), no. 1-4, 335–380.

\bibitem{APS} Atiyah, M. F. ,  Patodi, V. K. , Singer, I. M. \emph{Spectral asymmetry and Riemannian geometry}, Bull. London Math. Soc. 5 (1973), 229–234.

\bibitem{ccs} Barth, W. ,  Hulek, K. ,  Peters, C.,  van de Ven, A.  \emph{Compact Complex Surfaces}, Springer-Verlag Berlin Heidelberg, (2004).

\bibitem{BHH} Barthel, G. ,  Hirzebruch,  F. , H\" ofer, T. \emph{Geradenkonfigurationen und algebraische Fl\" achen}, Vieweg, Braunschweig Wiesbaden, (1987). 

\bibitem{BY} Berndt, B. C. , Yeap, B. P.  \emph{Explicit evaluations and reciprocity theorems for finite trigonometric sums}, Advances in Applied Mathematics 29 (2002), 358--385.

\bibitem{BKN} Bogya, N. , Korchm\'aros, G. , Nagy, G. P. \emph{Classification of k-nets}, European J. Combin. 48 (2015), 177-–185.

\bibitem{DS} Denham, G. ,  Suciu, A.  \emph{Multinets, parallel connections, and Milnor fibrations of arrangements}, Proc. Lond. Math. Soc. (3) 108 (2014), no. 6, 1435-–1470.

\bibitem{Dim} Dimca, A. \emph{Hyperplane arrangements. An introduction}, Universitext. Springer, Cham, (2017).

\bibitem{DMWZ} Dunn, C. , Miller, M. ,  Wakefield, M. , Zwicknagl, S. \emph{ Equivalence classes of Latin squares and nets in $\mathbb{CP}^{2}$}, Ann. Fac. Sci. Toulouse Math. (6) 23 (2014), no. 2, 335-–351.

\bibitem{FY} Falk, M. ,  Yuzvinsky, S. \emph{Multinets, resonance varieties, and pencils of plane curves}, Compos. Math. 143 (2007), no. 4, 1069–-1088.

\bibitem{gao} Gao, Y. \emph{A note on finite abelian covers}, Sci. China Math. (2011) 54: 1333.

\bibitem{GS} Gompf, R. E., Stipsicz, A. I., \emph{4-manifolds and Kirby calculus}, American Mathematical Society, Providence, (1999).

\bibitem{Gun} G\"{u}nt\"{u}rk\"{u}n, M.  H. \emph{Using tropical degenerations for proving the nonexistence of certain nets}, Ph.D. Thesis, Middle East Technical University, (2010). 

\bibitem{hirz} Hirzebruch, F. \emph{Arrangements of Lines and Algebraic Surfaces}, Arithmetic and Geometry. Progress in Mathematics, II (36). Birkh\" auser, Boston, (1983). 

\bibitem{hirz2} Hirzebruch, F. \emph{Hilbert Modular Surfaces}, L'Enseignement Math\'ematique, 
vol. 19 (1973), 183--281

\bibitem{KNP} Korchm\'aros, G. , Nagy, G.P. , Pace, N. \emph{$k$-nets embedded in a projective plane over a field}, Combinatorica 35 (2015), no. 1, 63--74.

\bibitem{LY} Libgober, A. ,  Yuzvinsky, S. \emph{Cohomology of the Orlik-Solomon algebras and local systems}, Compositio Math. 121 (2000), no. 3, 337-–361.

\bibitem{Oz} \"Ozba\u gc\i, B. \emph{Signatures of Lefschetz fibrations}, Pacific J. Math. 202 (2002), no. 1, 99--118. 

\bibitem{PY} Pereira,  J. V. ,  Yuzvinsky, S. \emph{Completely reducible hypersurfaces in a pencil}, Adv. Math. 219 (2008), no. 2, 672–-688. 

\bibitem{RU} Roulleau, X. , Urz\'ua, G. \emph{Chern slopes of simply connected complex surfaces of general type are dense in [2,3]}, Ann. of Math. (2) 182 (2015), no. 1, 287–306.

\bibitem{Sti} Stipins, J. \emph{On finite k-nets in the complex projective plane}, Ph.D. Thesis, University of Michigan (2007). 

\bibitem{Urz1} Urz\'ua, G. \emph{Arrangements of curves and algebraic surfaces}, J. Algebraic Geom. 19 (2010), no. 2, 335–365.

\bibitem{Urz2} Urz\'ua, G. \emph{On line arrangements with applications to 3-nets}, Adv. Geom. 10 (2010), no. 2, 287–310.

\bibitem{Urz3}  Urz\'ua, G. \emph{Arrangements of curves and algebraic surfaces}, Ph.D. Thesis, University of Michigan (2008).

\bibitem{Yuz1} Yuzvinsky, S. \emph{Realization of finite abelian groups by nets in $\mathbb{P}^{2}$},  Compos. Math. 140 (2004), no. 6, 1614–-1624.

\bibitem{Yuz2}  Yuzvinsky, S. \emph{A new bound on the number of special fibers in a pencil of curves}, Proc. Amer. Math. Soc. 137 (2009), no. 5, 1641–-1648.  

\bibitem{Yuz3} Yuzvinsky, S. \emph{Resonance varieties of arrangement complements}, Arrangements of hyperplanes—Sapporo 2009, 553–570, Adv. Stud. Pure Math., 62, Math. Soc. Japan, Tokyo, 2012.




\end{thebibliography}

\end{document}